\documentclass[12pt]{article}
\usepackage{color}
\usepackage{amsthm}
 \usepackage{amsmath}
 \usepackage{amssymb}
\usepackage{amsfonts}
\usepackage{amsbsy}
\usepackage{multicol}
\usepackage[dvips]{graphics}
\usepackage{graphics,pspicture}
\newcommand{\be}{\begin{equation}}
\newcommand{\ee}{\end{equation}}
\newcommand{\bea}{\begin{eqnarray}}
\newcommand{\eea}{\end{eqnarray}}
\newcommand{\ba}{\begin{array}}
\newcommand{\ea}{\end{array}}

\newcommand{\bc}{\begin{center}}
\newcommand{\ec}{\end{center}}
\newcommand{\ben}{\begin{enumerate}}
\newcommand{\een}{\end{enumerate}}
\newcommand{\bfi}{\begin{figure}}
\newcommand{\efi}{\end{figure}}

\newcommand{\bq}{\begin{quote}}
\newcommand{\eq}{\end{quote}}
\newcommand{\bqu}{\begin{quotation}}
\newcommand{\equ}{\end{quotation}}
\newenvironment{emphit}{\begin{itemize}}{\end{itemize}}
\newcommand{\bemp}{\begin{emphit}}
\newcommand{\eemp}{\end{emphit}}

\newcommand{\bt}{\begin{tabular}}
\newcommand{\et}{\end{tabular}}

\newtheorem{myth}{Theorem}[section]
\newtheorem{mylem}{Lemma}[section]
\newtheorem{mycor}{Corollary}[section]
\newtheorem{mypro}{Proposition}[section]
\newtheorem{mydef}{Definition}[section]
\newtheorem{myrem}{Remark}[section]

\begin{document}
\date{}
\title{Holomorphic Structure of Middle Bol Loops
}
\author{T. G. Jaiy\'e\d ol\'a\thanks{All correspondence to be addressed to this author.} \\
Department of Mathematics,\\
Obafemi Awolowo University,\\
Ile Ife 220005, Nigeria.\\
jaiyeolatemitope@yahoo.com\\tjayeola@oauife.edu.ng \and
S. P. David \\
Department of Mathematics,\\
Obafemi Awolowo University,\\
Ile Ife 220005, Nigeria.\\
davidsp4ril@yahoo.com\and
E. Ilojide \\
Department of Mathematics,\\
Federal University of Agriculture, \\
Abeokuta 110101, Nigeria.\\
emmailojide@yahoo.com\\ilojidee@unaab.edu.ng\and
Y. T. Oyebo\\
Department of Mathematics, \\Lagos State University, \\Ojo, Lagos State, Nigeria\\
oyeboyt@yahoo.com\\yakub.oyebo@lasu.edu.ng}\maketitle
\begin{abstract}
A loop $(Q,\cdot,\backslash,/)$ is called a middle Bol loop if it obeys the identity $x(yz\backslash x)=(x/z)(y\backslash x)$.
To every right (left) Bol loop corresponds a middle Bol loop via an isostrophism. In this paper, the structure of the holomorph of a middle Bol loop is explored. For some special types of automorphisms, the holomorph of a commutative loop is shown to be a commutative middle Bol loop if and only if the loop is a middle Bol loop and its automorphism group is abelian and a subgroup of both the group of middle regular mappings and the right multiplication group. It was found that commutativity (flexibility) is a necessary and sufficient condition for holomorphic invariance under the existing isostrophy between middle Bol loops and the corresponding right (left) Bol loops. The right combined holomorph of a middle Bol loop and its corresponding right (left) Bol loop was shown to be equal to the holomorph of the middle Bol loop if and only if the automorphism group is abelian and a subgroup of the multiplication group of the middle Bol loop. The obedience of an identity dependent on automorphisms was found to be a necessary and sufficient condition the left combined holomorph of a middle Bol loop and its corresponding left Bol loop to be equal to the holomorph of the middle Bol loop.
\end{abstract}
\section{Introduction}
\paragraph{}
Let $G$ be a non-empty set. Define a binary operation ($\cdot $) on
$G$. If $x\cdot y\in G$ for all $x, y\in G$, then the pair $(G, \cdot )$
is called a \textit{groupoid} or \textit{magma}.

If for any $a,b\in G$, each of the equations:
\begin{displaymath}
a\cdot x=b\qquad\textrm{and}\qquad y\cdot a=b
\end{displaymath}
has unique solutions in $G$ for $x$ and $y$ respectively, then $(G,
\cdot )$ is called a \textit{quasigroup}.

If there exists a unique element $e\in G$ called the
\textit{identity element} such that for all $x\in G$, $x\cdot
e=e\cdot x=x$, $(G, \cdot )$ is called a \textit{loop}. We write
$xy$ instead of $x\cdot y$, and stipulate that ($\cdot$) has lower
priority than juxtaposition among factors to be multiplied. For
instance, $x\cdot yz$ stands for $x(yz)$.

It can now be seen that a groupoid $(G, \cdot )$ is a quasigroup if
its left and right translation mappings are bijections or
permutations. Since the left and right translation mappings of a
loop are bijective, then the inverse mappings $L_x^{-1}$ and
$R_x^{-1}$ exist. Let
\begin{displaymath}
x\backslash y =yL_x^{-1}=y\mathcal{L}_x=x\mathcal{R}_y\qquad\textrm{and}\qquad
x/y=xR_y^{-1}=x\mathbb{R}_y=y\mathbb{L}_x
\end{displaymath}and note that
\begin{displaymath}
x\backslash y =z\Longleftrightarrow x\cdot
z=y\qquad\textrm{and}\qquad x/y=z\Longleftrightarrow z\cdot y=x.
\end{displaymath}
Hence, $(G, \backslash )$ and $(G, /)$ are also quasigroups. Using
the operations ($\backslash$) and ($/$), the definition of a loop
can be stated as follows.

\begin{mydef}\label{0:1}\textrm{
A \textit{loop} $(G,\cdot ,/,\backslash ,e)$ is a set $G$ together
with three binary operations ($\cdot $), ($/$), ($\backslash$) and
one nullary operation $e$ such that
\begin{description}
\item[(i)] $x\cdot (x\backslash y)=y$, $(y/x)\cdot x=y$ for all
$x,y\in G$,
\item[(ii)] $x\backslash (x\cdot y)=y$, $(y\cdot x)/x=y$ for all
$x,y\in G$ and
\item[(iii)] $x\backslash x=y/y$ or $e\cdot x=x$ for all
$x,y\in G$.
\end{description}}
\end{mydef}
We also stipulate that ($/$) and ($\backslash$) have higher priority
than ($\cdot $) among factors to be multiplied. For instance,
$x\cdot y/z$ and $x\cdot y\backslash z$ stand for $x(y/z)$ and
$x\cdot (y\backslash z)$ respectively.

In a loop $(G,\cdot )$ with identity element $e$, the \textit{left
inverse element} of $x\in G$ is the element $xJ_\lambda
=x^\lambda\in G$ such that
\begin{displaymath}
x^\lambda\cdot x=e
\end{displaymath}
while the \textit{right inverse element} of $x\in G$ is the element
$xJ_\rho =x^\rho\in G$ such that
\begin{displaymath}
x\cdot x^\rho=e.
\end{displaymath}
The group of all permutations on $G$ is called the permutation group of $G$ and denoted by $SYM(G)$.
The groups ${\cal M}_\rho(L,\cdot )=\Big<\{R_x,R_x^{-1}~:~x\in G\}\Big>$ and ${\cal M}(L,\cdot )=\Big<\{R_x,R_x^{-1},L_x,L_x^{-1}~:~x\in G\}\Big>$ are called the right multiplication group and multiplication group of $(G,\cdot )$ and, ${\cal M}_\rho(L,\cdot )\le {\cal M}(G,\cdot )\le SYM(G)$.

For an overview of the theory of
loops, readers may check \cite{davidref:7,davidref:8,davidref:9,davidref:16,davidref:10,davidref:12,davidref:21,davidref:22}.

A loop satisfying the identical relation
 \begin{equation}\label{davideq1}
 (xy\cdot z)y=x(yz\cdot y)
\end{equation}
is called a right Bol loop. A loop satisfying the identical relation
\begin{equation}\label{davideq2}
(x\cdot yx)z=x(y\cdot xz)
\end{equation}
is called a left Bol loop.

A loop $(Q,\cdot)$ is called a middle Bol if it satisfies the identity
\begin{equation}\label{davideq2.1}
x(yz\backslash x) = (x/z)(y\backslash x)
\end{equation}
It is known that the identity \eqref{davideq2.1} is universal under loop isotopy and that the universality of \eqref{davideq2.1} implies the power associativity of the middle Bol loops (Grecu and Syrbu \cite{davidref:65}). Furthermore, \eqref{davideq2.1} is a necessary and sufficient condition for the universality of the anti-automorphic inverse property (Syrbu \cite{davidref:6}). They were originally introduced in 1967 by Belousov \cite{davidref:24} and were later considered in 1971 by Gvaramiya \cite{davidref:23}, who proved that a loop $(Q,\circ)$ is middle Bol if and only if there exists a right Bol loop $(Q,\cdot)$ such that
\begin{equation}\label{davideq2.2}
x\circ y=(y\cdot xy^{-1})y,~\textrm{for every}~x,y\in Q.
\end{equation}
This result of Gvaramiya \cite{davidref:23} is formally stated below:
\begin{myth}\label{davidthm1}
If $(Q,\cdot)$ is a left (right) Bol loop then the groupoid $(Q,\circ)$, where $x\circ y=y(y^{-1}x\cdot y)$ (respectively, $x\circ y=(y\cdot xy^{-1})y$), for all $x,y \in Q$ , is a middle Bol loop and, conversely, if $(Q,\circ)$ is a middle Bol loop then there exists a left(right) Bol loop $(Q,\cdot)$ such that $x\circ y=y(y^{-1}x\cdot y)$ (respectively, $x\circ y=(y\cdot xy^{-1})y$), for all $x,y \in Q$.
\end{myth}
\begin{myrem}
Theorem~\ref{davidthm1} implies that if $(Q,\cdot)$ is a left Bol loop and $(Q,\circ)$ is the corresponding middle Bol loop then $x\circ y=x/y^{-1}$ and $x\cdot y=x//y^{-1}$ , where "$/$" ("$//$") is the right division in $(Q,\cdot)$ (respectively, in $(Q,\circ)$). Similarly, if $(Q,\cdot)$ is a right Bol loop and $(Q,\circ)$ is the corresponding middle Bol loop then $x\circ y=y^{-1}\backslash y$ and $x\cdot y=y//x^{-1}$ , where "$\backslash$" ("$//$") is the left (right) division in $(Q,\cdot)$ (respectively, in $(Q,\circ)$). Hence, middle Bol loops are isostrophs of left and right Bol loops.\\

If $(Q,\circ)$ is a middle Bol loop and $(Q,\cdot)$ is the corresponding left Bol loop, then $(Q,\ast)$, where $x\ast y=y\cdot x$, for every $x,y\in Q$, is the corresponding right Bol loop for $(Q,\circ)$. So, $(Q,\cdot)$ is a left Bol loop, $(Q,\ast)$ is a right Bol loop and
$$x\circ y=y(y^{-1} x\cdot y)=[y\ast(x\ast y^{-1})]\ast y,$$
for every $x,y\in Q$.
\end{myrem}

After then, middle Bol loops resurfaced in literature not until 1994 and 1996 when Syrbu \cite{davidref:4,davidref:5} considered them in-relation to the universality of the elasticity law.

In 2003, Kuznetsov (\cite{davidref:22}), while studying gyrogroups (a special class of Bol loops) established some algebraic properties of middle Bol loops and designed a method of constructing a middle Bol loop from a gyrogroup. According to him, in a middle Bol loop $(Q,\cdot)$ with identity element $e$, the following are true.
\begin{enumerate}
    \item The left inverse element $x^\lambda$ and the right inverse $x^\rho$ to an element $x\in Q$ coincide : $x^\lambda$=$x^\rho$.
  \item If $(Q,\cdot,e)$ is a left Bol loop and "$/$" is the right inverse operation to the operation $"\cdot "$ , then the operation  $x\circ y=x/y^{-1}$ is a middle Bol loop $(Q,\circ,e)$, and every one middle Bol loop can be obtained in a similar way from some left Bol loop.
\end{enumerate}
These confirm the observations of earlier authors mentioned above.

In 2010, Syrbu \cite{davidref:6} studied the connections between structure and properties of middle Bol loops and of the corresponding left Bol loops. It was noted that two middle Bol loops are isomorphic if and only if the corresponding left (right) Bol loops are isomorphic, and a general form of the autotopisms of middle Bol loops was deduced. Relations between different sets of elements, such as nucleus, left (right,middle) nuclei, the set of Moufang elements, the center, e.t.c. of a middle Bol loop and  left Bol loops were established.

In 2012, Grecu and Syrbu \cite{davidref:65} proved that two middle Bol loops are isotopic if and only if the corresponding right (left) Bol loops are isotopic. They also proved that a middle Bol loop $(Q,\circ)$ is flexible if and only if the corresponding right Bol loop $(Q,\cdot)$ satisfies the identity
$$(yx)^{-1}\cdot \big(x^{-1}\cdot y^{-1}\big)^{-1}x=x.$$

In 2012, Drapal and Shcherbacov \cite{davidref:1} rediscovered the middle Bol identities in a new way.

In 2013,  Syrbu and Grecu \cite{davidref:66n} established a necessary and sufficient condition  for the quotient loops of a  middle Bol loop and of its corresponding right Bol loop to be isomorphic.

In 2014, Grecu and Syrbu \cite{davidref:66} established:
\begin{enumerate}
  \item that the commutant (centrum) of a middle Bol loop is an AIP-subloop and
  \item a necessary and sufficient condition when the commutant is an invariant under the existing isostrophy between middle Bol loop and the corresponding right Bol loop.
\end{enumerate}

In 1994, Syrbu \cite{davidref:4}, while studying loops with universal elasticity $(xy\cdot x=x\cdot yx)$ established a necessary and sufficient condition  $(xy/z)(b\backslash xz)=x(b\backslash[(by/z)(b\backslash xz)])$ for a loop  $(Q,\cdot,\backslash,/)$ to be universally elastic. Furthermore, she constructed some finite examples of loops in which this condition and the middle Bol identity $x(yz\backslash x)=(x/z)(y\backslash x)$ are equivalent, and then posed an open problem of investigating if these two identities are also equivalent in all other finite loops.

In 2012, Drapal and Shcherbacov \cite{davidref:1} reported that Kinyon constructed a non-flexible middle Bol loop of order $16$. In 2015, we discovered new algebraic properties of middle Bol loops in Jaiy\'e\d ol\'a et.al. \cite{davidref:david1}.

\paragraph{}
Interestingly, Adeniran \cite{phd79} and
Robinson \cite{25}, Chiboka and Solarin
\cite{phd80}, Bruck \cite{7}, Bruck and Paige \cite{phd40},
Robinson \cite{phd7}, Huthnance \cite{phd44} and  Adeniran \cite{phd79} have respectively studied the holomorphs of Bol loops,
conjugacy closed loops, inverse property loops,
A-loops, extra loops, weak inverse property loops and Bruck loops. A set of results on the holomorph of some varieties of loops can be found in Jaiyeola \cite{davidref:10}. The latest studies on the holomorph of generalized Bol loops can be found in Adeniran et. al. \cite{ho1} and Jaiyeola and Popoola \cite{ho2}. Isere et. al. \cite{isere1} studied the holomorphy of Osborn loops.

\paragraph{}In this paper, our objective is to explore holomorphic structure of a middle Bol loop. Before this, we shall take some definitions and state some important results which will be often used.

For any quasigroup $(Q,\cdot)$, the group of autotopisms under componentwise composition is given by
$$ATP(Q,\cdot)=\{T=(U,V,W)\in SYM(Q)^3~|~xU\cdot yV=(x\cdot y)W~\forall ~x,y\in Q\}.$$
If $U=V=W$, then $U$ is called an automorphism of $(Q,\cdot)$ and the set of such mappings forms a group $AUT(Q,\cdot)$ called the automorphism group of $(Q,\cdot)$. We now introduce the set of anti-autotopisms given by
$$AATP(Q,\cdot)=\{T'=<U',V',W'>\in ~SYM(Q)^3~|~xU'\cdot yV'=(y\cdot x)W'~\forall ~x,y\in Q\}.$$

\begin{mydef}
Let $(G,\cdot )$ be a quasigroup. Then
\begin{enumerate}
\item a bijection $U$ is called autotopic if there exists $(U,V,W)\in ATP(G,\cdot )$; the set of all such mappings forms a group $\Sigma(G,\cdot )$.
\item a bijection $U$ is called middle regular ($\mu$-regular) if there exists a bijection $U'$ such that $(U,U'^{-1},I)\in ATP(G,\cdot )$. $U'$ is called the adjoint of $U$. The set of all $\mu$-regular mappings forms a group $\Phi(G,\cdot )\le\Sigma(G,\cdot )$. The set of all adjoint mapping
    forms a group $\Psi(G,\cdot )$.
\end{enumerate}
\end{mydef}

\begin{mydef}
Let $(L,\cdot)$ be a loop. The pair $(H, \circ )=H(L,\cdot)$ given by
$$H=A(L)\times L~\textrm{where}~A(L)\leq AUT(L, \cdot )$$
$$~\textrm{such that}~(\alpha , x)\circ (\beta , y)=(\alpha \beta , x\beta \cdot y)$$
for all $(\alpha , x), (\beta , y)\in H$ is called the $A(L)$-Holomorph of $(L, \cdot )$.
\end{mydef}
We shall need the following results.
\begin{mylem}\label{davidlem4.11a}
Let $(L,\cdot)$ be a loop. An $A(L)$-Holomorph $(H, \circ )=H(L,\cdot)$ of $(L, \cdot )$ is a loop.
\end{mylem}

\begin{mypro}(Grecu and Syrbu \cite{davidref:65})\label{davidpro7}

Let $(Q,\circ)$ be a middle Bol loop and let $(Q,\cdot)$ be a corresponding right(left) Bol loop. Then, $ATP(Q,\cdot)=ATP(Q,\circ)$.
\end{mypro}

\section{Main Results}
\subsection{Holomorph of a Middle Bol Loop}
\begin{mylem}\label{davidlem16}
Let $(L,\cdot,/,\backslash)$ be a loop with holomorph $H(L,\cdot)$. The following are equivalent:
\begin{description}
  \item[(a)]$H(L,\cdot)$ is a middle Bol loop.
  \item[(b)]$(x\delta)\cdot(y\cdot z\delta)\backslash x=(x/z)\delta\cdot y\backslash x ~\forall ~x,y,z\in L, ~\delta\in A(L)$.
  \item[(c)]$(x\delta)\cdot(y\cdot z\delta)\backslash x=(x\delta/z\delta)\cdot y\backslash x~ \forall ~x,y,z\in L, ~\delta\in A(L)$.
  \item[(d)]$\langle\delta^{-1}\mathbb{L}_x\delta,\mathcal{R}_x,\mathcal{R}_xL_{x\delta}\rangle\in AATP(L,\cdot)~\forall ~x\in L, ~\delta\in A(L)$.
\end{description}
\end{mylem}
\begin{proof}
We desire a necessary and sufficient condition for
$$x(yz\backslash x)=(x/z)(y\backslash x)$$
$$\textrm{Now},~(\alpha, x)\circ(\beta, y)=(\gamma, z)\Rightarrow (\gamma, z)/(\beta, y)=(\alpha, x),$$
$$\textrm{so},~(\alpha\beta, x\beta\cdot y)=(\gamma, z)$$
$$\Rightarrow \alpha = \gamma\beta^{-1}, x= (z/y)\beta^{-1}$$
\begin{equation}\label{davideq19a}
\therefore (\gamma,z)/(\beta, y)= \Big(\gamma\beta^{-1},(z/y)\beta^{-1}\Big)
\end{equation}
$$\textrm{Also},~(\alpha,x)\circ(\beta,y)=(\gamma, z)\Rightarrow (\beta,y)=(\alpha,x)\backslash (\gamma,z).$$
$$\textrm{Thus},~(\alpha\beta, x\beta\cdot y)=(\gamma, z)$$
$$\Rightarrow \beta=\alpha^{-1}\gamma, y=(x\alpha^{-1}\gamma)\backslash z$$
\begin{equation}\label{davideq19b}
\therefore (\beta,y)=\Big(\alpha^{-1}\gamma,(x\alpha^{-1}\gamma)\backslash z\Big)=(\alpha,x)\backslash(\gamma,z)
\end{equation}
We want to find a necessary and sufficient condition for $$(\alpha,x)\circ[(\beta,y)\circ(\gamma,z)\backslash (\alpha,x)]=[(\alpha,x)/(\gamma,x)]\circ[(\beta,y)\backslash (\alpha,x)]$$
$$LHS=(\alpha,x)\circ[(\beta,y)\circ(\gamma,z)\backslash (\alpha,x)]$$
$$=(\alpha,x)\circ\Big((\beta\gamma)^{-1}\alpha, (y\gamma\cdot z)\gamma^{-1}\beta^{-1}\alpha\backslash x\Big)$$
$$=(\alpha,x)\circ\Big((\beta\gamma)^{-1}\alpha, y\beta^{-1}\alpha\cdot z\gamma^{-1}\beta^{-1}\alpha\backslash x\Big).$$
$$\therefore LHS=\Big(\alpha\gamma^{-1}\beta^{-1}\alpha, (x\gamma^{-1}\beta^{-1}\alpha)\cdot(y\beta^{-1}\alpha\cdot z\gamma^{-1}\beta^{-1}\alpha)\backslash x\Big).$$
$$RHS=[(\alpha,x)/(\gamma,x)]\circ[(\beta,y)\backslash (\alpha,x)]$$
$$=\Big(\alpha\gamma^{-1}, (x/z)\gamma^{-1}\Big)\circ\Big(\beta^{-1}\alpha, (y\beta^{-1}\alpha)\backslash x\Big).$$
$$\therefore RHS=\Big(\alpha\gamma^{-1}\beta^{-1}\alpha,(x/z)\gamma^{-1}\beta^{-1}\alpha\cdot(y\beta^{-1}\alpha)\backslash x\Big).$$
$$LHS=RHS\Leftrightarrow$$
$$(x\gamma^{-1}\beta^{-1}\alpha)\cdot(y\beta^{-1}\alpha\cdot z\gamma^{-1}\beta^{-1}\alpha)\backslash x=(x/z)\gamma^{-1}\beta^{-1}\alpha\cdot(y\beta^{-1}\alpha)\backslash x.$$
Let $\delta=\gamma^{-1}\beta^{-1}\alpha$, then
$$(x\delta)\cdot(y\gamma\delta\cdot z\delta)\backslash x=(x/z)\delta\cdot(y\gamma\delta)\backslash x$$
Replacing $y$ by $y(\gamma\delta)^{-1}$, we get
$$(x\delta)\cdot(y(\gamma\delta)^{-1}\gamma\delta\cdot z\delta)\backslash x=(x/z)\delta\cdot(y(\gamma\delta)^{-1}\gamma\delta)\backslash x$$
\begin{equation}\label{davideq20a}
\iff(x\delta)\cdot(y\cdot z\delta)\backslash x=(x/z)\delta\cdot y\backslash x=(x\delta/z\delta)\cdot y\backslash x
\end{equation}
 $$\textrm{Note that}~y\delta=(y/z)\delta\cdot z\delta \Rightarrow y\delta/z\delta=(y/z)\delta.$$
By writing equation \eqref{davideq20a} in translation form, we have:
$$(x\delta)\cdot(y\cdot z\delta)\mathcal{R}_x=(z\mathbb{L}_x)\delta\cdot y\mathcal{R}_x$$
$$\iff(y\cdot z\delta)\mathcal{R}_xL_{x\delta}=z\mathbb{L}_x\delta\cdot y\mathcal{R}_x.$$
Replacing $z$ by $z\delta^{-1}$, we have;
$$(y\cdot z\delta^{-1}\delta)\mathcal{R}_xL_{x\delta}=z\delta^{-1}\mathbb{L}_x\delta\cdot y\mathcal{R}_x$$
$$\iff(y\cdot z)\mathcal{R}_xL_{x\delta}=z\delta^{-1}\mathbb{L}_x\delta\cdot y\mathcal{R}_x$$
$$\iff \langle\delta^{-1}\mathbb{L}_x\delta,\mathcal{R}_x,\mathcal{R}_xL_{x\delta}\rangle\in AATP(L,\cdot).$$
\end{proof}
\begin{mylem}\label{davidlem17}
Let $(L,\cdot,/,\backslash)$ be a loop with holomorph $H(L,\cdot)$. Then $H(L,\cdot)$ is a commutative loop if and only if $A(L,\cdot)$ is an abelian group and $\langle\beta,\alpha^{-1},I\rangle\in AATP(L,\cdot)~ \forall ~\alpha,\beta\in A(L)$.
\end{mylem}
\begin{proof}
$$\textrm{Let}~(\alpha,x)\circ(\beta,y)=(\alpha\beta,x\beta\cdot y)=LHS~\textrm{and}$$
$$(\beta,y)\circ(\alpha,x)=(\beta\alpha,y\alpha\cdot x)=RHS.$$
$$\textrm{Then},~LHS=RHS \Leftrightarrow$$
$$(\alpha\beta,x\beta\cdot y)=(\beta\alpha,y\alpha\cdot x)\Leftrightarrow$$
$$\alpha\beta=\beta\alpha,x\beta\cdot y\alpha^{-1}=yx\Leftrightarrow$$
$$\alpha\beta=\beta\alpha~\textrm{and}~\langle\beta,\alpha^{-1},I\rangle\in AATP(L,\cdot).$$
\end{proof}
\begin{mycor}\label{davidcor23a}
Let $(L,\cdot,/,\backslash)$ be a commutative loop with holomorph $H(L,\cdot)$. Then $(H,\circ )=H(L,\cdot)$ is a commutative loop if and only if $A(L,\cdot)$ is an abelian group and $(\beta,\alpha^{-1},I)\in ATP(L,\cdot)~\forall~\alpha,\beta\in A(L)$.
\end{mycor}
\begin{proof}
Use Lemma~\ref{davidlem17}.
\end{proof}
\begin{mycor}\label{davidcor23b}
Let $(L,\cdot,/,\backslash)$ be a commutative loop with holomorph $H(L,\cdot)$. Then $(H,\circ )=H(L,\cdot)$ is a commutative loop if and only if $A(L,\cdot)$ is an abelian group and $\beta\in\Phi(L,\cdot),~\beta'=\alpha\in\Psi(L,\cdot)$ for each $(\alpha,\beta)\in A(L)^2$.
\end{mycor}
\begin{proof}
Use Corollary~\ref{davidcor23a}.
\end{proof}

\begin{mylem}\label{davidlem17a}
Let $(L,\cdot,/,\backslash)$ be a commutative middle Bol loop with an holomorph $(H,\circ )=H(L,\cdot)$. If:
\begin{enumerate}
  \item $\delta=\delta(x,z)=R_{(z\backslash x)}R^{-1}_z$ for each $\delta\in A(L)$ and for arbitrarily fixed $x,z\in L$; and
  \item $\mathcal{R}^{-1}_wR_{y\delta}=R_yR^{-1}_wR_{w\delta}~\forall~ y,w\in L$ and $\delta\in A(L)$,
\end{enumerate}
 then $H(L,\cdot)$ is a middle Bol loop.
\end{mylem}
\begin{proof}
Observe that
$$\langle\delta^{-1}\mathbb{L}_x\delta,\mathcal{R}_x,\mathcal{R}_xL_{x\delta}\rangle=
(\delta^{-1},\mathcal{R}_x,I)\circ\langle\mathbb{L}_x,\mathcal{R}_x,\mathcal{R}_xL_x\rangle\circ(\delta,\mathcal{R}^{-1}_x,L^{-1}_xL_{x\delta}).$$
$$\textrm{On one hand},~(\delta^{-1},\mathcal{R}_x,I)\in ATP(L,\cdot)\Leftrightarrow y\delta^{-1}\cdot z\mathcal{R}_x=yz$$
$$\Leftrightarrow y\delta^{-1}(z\backslash x)=yR_z\Leftrightarrow \delta^{-1}R_{(z\backslash x)}=R_z\Leftrightarrow \delta =\delta(x,z)=R_{(z\backslash x)}R^{-1}_z.$$
$$\textrm{On another hand},~(\delta,\mathcal{R}^{-1}_w,L^{-1}_wL_{w\delta})\in ATP(L,\cdot)\Leftrightarrow$$
$$y\delta\cdot v\mathcal{R}^{-1}_w=(yv)L^{-1}_wL_{w\delta}\Leftrightarrow v\mathcal{R}^{-1}_wL_{y\delta}=vL_yL^{-1}_wL_{w\delta}$$
$$\Leftrightarrow \mathcal{R}^{-1}_wL_{y\delta}=L_yL^{-1}_wL_{w\delta}\Leftrightarrow \mathcal{R}^{-1}_wR_{y\delta}=R_yR^{-1}_wR_{x\delta}.$$
\end{proof}
\begin{mylem}\label{davidthm5}
Let $(L,\cdot,/,\backslash)$ be a commutative loop such that $\mathcal{R}^{-1}_wR_{y\delta}=R_yR^{-1}_wR_{w\delta},~\forall~ w,y\in L$ and $\delta\in A(L)$. $(H,\circ )=H(L,\cdot)$ is a middle Bol loop if and only if $(L,\cdot)$ is a middle Bol loop and $\delta =\delta(x,z)=R_{(z\backslash x)}R^{-1}_z$ for each $\delta\in A(L)$ and for arbitrarily fixed $x,z\in L$.
\end{mylem}
\begin{proof}
This is similar to Lemma~\ref{davidlem17a}.
\end{proof}

\begin{myth}\label{davidthm5a}
Let $(L,\cdot,/,\backslash)$ be a commutative loop such that $\mathcal{R}^{-1}_wR_{y\delta}=R_yR^{-1}_wR_{w\delta},~\forall~ w,y\in L$ and $\delta\in A(L)$. $(H,\circ )=H(L,\cdot)$ is a commutative middle Bol loop if and only if
\begin{enumerate}
  \item $(L,\cdot)$ is a middle Bol loop;
  \item $A(L)$ is an abelian group;
  \item $\delta=\delta(x,z)=R_{(z\backslash x)}R^{-1}_z$ for each $\delta\in A(L)$ and for arbitrarily fixed $x,z\in L$; and
  \item $\beta\in\Phi(L,\cdot),~\beta'=\alpha\in\Psi(L,\cdot)$ for each $(\alpha,\beta)\in A(L)^2$.
\end{enumerate}
\end{myth}
\begin{proof}
Use Lemma~\ref{davidthm5} and Corollary~\ref{davidcor23b}.
\end{proof}

\begin{mycor}\label{davidthm5a.1}
Let $(L,\cdot,/,\backslash)$ be a commutative loop such that $\mathcal{R}^{-1}_wR_{y\delta}=R_yR^{-1}_wR_{w\delta},~\forall~ w,y\in L$ and $\delta\in A(L)$. $(H,\circ )=H(L,\cdot)$ is a commutative middle Bol loop if and only if $(L,\cdot)$ is a middle Bol loop, $A(L)\le\Phi(L,\cdot)\cap\mathcal{M}_\rho(L,\cdot)$
and $A(L)$ is an abelian group.
\end{mycor}
\begin{proof}
Apply Theorem~\ref{davidthm5a}.
\end{proof}

\begin{mylem}\label{davidthm6}
Let $(L,\cdot,/,\backslash)$ be a commutative loop such that $\delta=\delta(x,z)=R_{(z\backslash x)}R^{-1}_z$ for each $\delta\in A(L)$ and for arbitrarily fixed $x,z\in L$. $(H,\circ )=H(L,\cdot)$ is a middle Bol loop if and only if $(L,\cdot)$ is a middle Bol loop and $\mathcal{R}^{-1}_wR_{y\delta}=R_yR^{-1}_wR_{w\delta} ~\forall ~w,y\in L$ and $\delta\in A(L)$.
\end{mylem}
\begin{proof}
This is similar to the proof of Lemma~\ref{davidthm5}.
\end{proof}

\begin{myth}\label{davidthm5b}
Let $(L,\cdot,/,\backslash)$ be a commutative loop such that $\delta=\delta(x,z)=R_{(z\backslash x)}R^{-1}_z$ for each $\delta\in A(L)$ and for arbitrarily fixed $x,z\in L$. $(H,\circ )=H(L,\cdot)$ is a commutative middle Bol loop if and only if
\begin{enumerate}
  \item $(L,\cdot)$ is a middle Bol loop;
  \item $A(L)$ is an abelian group;
  \item $\mathcal{R}^{-1}_wR_{y\delta}=R_yR^{-1}_wR_{w\delta}$ for all $w,y\in L$ and $\delta\in A(L)$; and
  \item $\beta\in\Phi(L,\cdot),~\beta'=\alpha\in\Psi(L,\cdot)$ for each $(\alpha,\beta)\in A(L)^2$.
\end{enumerate}
\end{myth}
\begin{proof}
Use Lemma~\ref{davidthm6} and Corollary~\ref{davidcor23b}.
\end{proof}

\begin{mycor}\label{davidthm5b.1}
Let $(L,\cdot,/,\backslash)$ be a commutative loop such that $A(L)\in\mathcal{M}_\rho(L,\cdot)$. $(H,\circ )=H(L,\cdot)$ is a commutative middle Bol loop if and only if $(L,\cdot)$ is a middle Bol loop, $\mathcal{R}^{-1}_wR_{y\delta}=R_yR^{-1}_wR_{w\delta},~\forall~ w,y\in L$ and $\delta\in A(L)$, $A(L)\le\Phi(L,\cdot)$ and $A(L)$ is an abelian group.
\end{mycor}
\begin{proof}
Apply Theorem~\ref{davidthm5b}.
\end{proof}
\subsection{Holomorph of a Middle Bol Loop and its Corresponding Right Bol Loop}
\begin{myth}\label{davidthm7}
Let $(Q,\cdot)$ be a right Bol loop with holomorph $(H,\odot)$ and $(Q,\ast)$ its corresponding middle Bol loop with holomorph $(H', \circledast)$. Then $(H, \odot)=(H', \circledast)$ if and only if $(Q,\cdot)$ is commutative.
\end{myth}
\begin{proof}
Let $(Q,\cdot)$ be a right Bol loop and $(Q,\ast)$ its corresponding middle Bol loop, then
\begin{equation}\label{33} x\ast y=(y\cdot xy^{-1})y
\end{equation}
$$(H,\odot)=H(Q,\cdot):(\alpha,x)\odot(\beta,y)=(\alpha\beta, x\beta\cdot y)~\textrm{and}$$
$$(H',\circledast)=H'(Q,\ast): (\gamma,x)\circledast(\delta,y)=(\gamma\delta, x\delta\ast y)$$
$$=(\gamma\delta, (y\cdot x\delta y^{-1})y).$$
Note that: $H=A(Q,\cdot)\times(Q,\cdot)$. By Proposition~\ref{davidpro7}, $AUT(Q,\cdot)=AUT(Q, \ast)$.
So, $$H'=A(Q,\ast)\times(Q,\ast)=A(Q,\cdot)\times(Q,\ast)=H.$$
Therefore, $$(H,\odot)=H(Q,\cdot):(\alpha,x)\odot(\beta,y)=(\alpha\beta, x\beta\cdot y)~\textrm{and}$$
 $$(H,\circledast)=H(Q,\ast): (\alpha,x)\circledast(\beta,y)=(\alpha\beta, x\beta\ast y)$$
\begin{equation}\label{34}
  =\Big(\alpha\beta, (y\cdot x\beta y^{-1})y\Big)
\end{equation}
Thus, $(H,\odot)=(H,\circledast)\iff $
$$x\beta\cdot y=(y\cdot x\beta y^{-1})y\iff$$
$$x\beta=y\cdot x\beta y^{-1}\iff y\backslash x\beta=x\beta y^{-1}.$$
In translation form, we have:
$$x\beta=(x\beta y^{-1})L_y\iff x\beta=x\beta R^{-1}_yL_y\iff$$
$$\beta=\beta R^{-1}_yL_y\iff R_y=L_y.$$
$$\therefore a\cdot y= y\cdot a.$$
\end{proof}
\begin{myrem}\label{davidrem6a}
Theorem~\ref{davidthm7} means that commutativity is a necessary and sufficient condition for holomorphic invariance under existing isostrophy between middle Bol loops and the corresponding right Bol loops.
\end{myrem}
\subsection{Right Combined Holomorph of a Middle Bol Loop}
Let $(Q,\cdot)$ and $(Q, \ast)$ be two loops with corresponding $A(Q)$-holomorphs $H(Q,\cdot)=(H,\odot)$ and $A'(Q,\ast)$-holomorph $H'(Q,\ast)=(H,\circledast)$. Let $\mathcal{H}=A(Q,\cdot)\bigcap A'(Q,\ast)\times Q$ and define $(\ast,\cdot)$ on $\mathcal{H}$ such that $\forall~ (\alpha, x), (\beta, y)\in \mathcal{H}$,
$$(\alpha,x)(\ast,\cdot)(\beta,y)=\{(\beta,y)\odot[(\alpha,x)\odot(\beta,y)^{-1}]\}\odot(\beta,y)$$
$(\mathcal{H},(\ast,\cdot))$ is called the right combined holomorph of $H(Q,\cdot)$ and $H'(Q,\ast)$.

\begin{myth}\label{davidthm8}
Let $(Q,\cdot)$ be a right Bol loop with holomorph $(H,\odot)$ and  its corresponding middle Bol loop $(Q,\ast)$ with holomorph $(H, \circledast)$. If $\Big(\mathcal{H},(\ast,\cdot)\Big)$ is the right combined holomorph of $H(Q,\cdot)$ and $H'(Q,\ast)$, then $\Big(\mathcal{H},(\ast,\cdot)\Big)=(H', \circledast)$ if and only if $\beta=\beta(y)=R^{-1}_yL_{y\alpha}L^{-1}_yR_y$ for any arbitrarily fixed $y\in Q$ and $\alpha\beta=\beta\alpha$ for all $\alpha,\beta\in A(Q,\cdot)$.
\end{myth}

\begin{proof}
Let $(Q,\cdot)$ be a right Bol loop and $(Q,\ast)$ its corresponding middle Bol loop, then by \eqref{33},
$x\ast y=(y\cdot xy^{-1})y$. Following the argument from \eqref{33} to \eqref{34}, we have
$$(H,\odot)=H(Q,\cdot):(\alpha,x)\odot(\beta,y)=(\alpha\beta, x\beta\cdot y)~\textrm{and}$$
 $$(H,\circledast)=H(Q,\ast): (\alpha,x)\circledast(\beta,y)=(\alpha\beta, x\beta\ast y)$$
$$=\Big(\alpha\beta, (y\cdot x\beta y^{-1})y\Big).$$
It is easy to check that
\begin{equation}\label{35}
 (\beta,y)^{-1}=\Big(\beta^{-1},(y^{-1})\beta^{-1}\Big)
\end{equation}
Thus, $(\alpha,x)(\circledast,\cdot)(\beta,y)=\{(\beta,y)\odot[(\alpha,x)\odot(\beta,y)^{-1}]\}\odot(\beta,y)$
$$=\{(\beta,y)\odot(\alpha\beta^{-1}, x\beta^{-1}\cdot(y^{-1})\beta^{-1})\}\odot(\beta,y)$$
$$=(\beta\alpha\beta^{-1}\beta, [y\alpha\beta^{-1}(x\beta^{-1}\cdot(y^{-1})\beta^{-1})]\beta\cdot y)$$
$$=(\beta\alpha,[y\alpha(x\beta^{-1}\beta\cdot(y^{-1}))]\cdot y)$$
$$=\Big(\beta\alpha,[y\alpha\cdot(xy^{-1})]y\Big).$$
Therefore,
$$\Big(\mathcal{H},(\ast,\cdot)\Big)=(H', \circledast)\Leftrightarrow (\alpha,x)(\circledast,\cdot)(\beta,y)=(\alpha,x)\circledast(\beta,y)\iff$$
$$\Big(\alpha\beta, (y\cdot x\beta y^{-1})y\Big)=\Big(\beta\alpha,[y\alpha\cdot(xy^{-1})]y\Big)\iff$$
$$\alpha\beta=\beta\alpha, (y\cdot x\beta y^{-1})y=[y\alpha\cdot(xy^{-1})]y)\iff$$
$$\alpha\beta=\beta\alpha,\beta R^{-1}_yL_y=R^{-1}_yL_{y\alpha}\iff$$
$$\alpha\beta=\beta\alpha,\beta=R^{-1}_yL_{y\alpha}L^{-1}_yR_y.$$
\end{proof}
\begin{mycor}\label{davidthm8.1}
Let $(Q,\cdot)$ be a right Bol loop with holomorph $(H,\odot)$ and  its corresponding middle Bol loop $(Q,\ast)$ with holomorph $(H, \circledast)$. If $\Big(\mathcal{H},(\ast,\cdot)\Big)$ is the right combined holomorph of $H(Q,\cdot)$ and $H'(Q,\ast)$, then $\Big(\mathcal{H},(\ast,\cdot)\Big)=(H', \circledast)$ if and only if $A(Q,\cdot)\le\mathcal{M}(Q,\cdot)$ and $A(Q,\cdot)$ is abelian.
\end{mycor}
\begin{proof}
Apply Theorem~\ref{davidthm8}.
\end{proof}
\subsection{Holomorph of a Middle Bol Loop and its Corresponding Left Bol Loop}
\begin{myth}\label{davidthm7l}
Let $(Q,\cdot)$ be a left Bol loop with holomorph $(H,\odot)$ and $(Q,\ast)$ its corresponding middle Bol loop with holomorph $(H', \circledast)$. Then $(H, \odot)=(H', \circledast)$ if and only if $(Q,\cdot)$ is flexible.
\end{myth}
\begin{proof}
Let $(Q,\cdot)$ be a left Bol loop and $(Q, \ast)$ its corresponding middle Bol loop,
\begin{equation}\label{36}
  x\ast y=y(y^{-1}x\cdot y)
\end{equation}
$$(H,\odot)=H(Q,\cdot):(\alpha,x)\odot(\beta,y)=(\alpha\beta, x\beta\cdot y)~\textrm{and}$$
$$(H',\circledast)=H'(Q,\ast): (\gamma,x)\circledast(\delta,y)=(\gamma\delta, x\delta\ast y)$$
$$=\Big(\gamma\delta, y(y^{-1}x\delta\cdot y)\Big).$$
Note that: $H=A(Q,\cdot)\times(Q,\cdot)$. By Proposition~\ref{davidpro7}, $AUT(Q,\cdot)=AUT(Q, \ast)$.
So, $$H'=A(Q,\ast)\times(Q,\ast)=A(Q,\cdot)\times(Q,\ast)=H.$$
Therefore, $(H,\odot)=H(Q,\cdot):(\alpha,x)\odot(\beta,y)=(\alpha\beta, x\beta\cdot y)~\textrm{and}$
 $$(H,\circledast)=H(Q,\ast): (\alpha,x)\circledast(\beta,y)=(\alpha\beta, x\beta\ast y)$$
 \begin{equation}\label{37}
   =\Big(\alpha\beta, y(y^{-1}x\beta\cdot y)\Big)
 \end{equation}
Thus, $$(H,\odot)=(H,\circledast)\iff x\beta\cdot y=y(y^{-1}x\beta\cdot y)\iff$$
$$\beta R_yL_{y^{-1}}=\beta L_{y^{-1}}R_y\iff L_yR_y=R_yL_y\iff (y\cdot a)y=y(a\cdot y).$$
\end{proof}
\begin{myrem}\label{davidrem7a}
Theorem~\ref{davidthm7l} means flexibility is a necessary and sufficient condition for holomorphic invariance under existing isostrophy between middle Bol loops and the corresponding left Bol loops.
\end{myrem}
\subsection{Left Combined Holomorph of a Middle Bol Loop}
Let $(Q,\cdot)$ and $(Q, \ast)$ be two loops with corresponding $A(Q)$-holomorphs $H(Q,\cdot)=(H,\odot)$ and $A'(Q,\ast)$-holomorph $H'(Q,\ast)=(H,\circledast)$. Let $\mathcal{H}=A(Q,\cdot)\bigcap A'(Q,\ast)\times Q$ and define $(\ast,\cdot)$ on $\mathcal{H}$ such that
$\forall~ (\alpha, x), (\beta, y)\in \mathcal{H}$,
$$(\alpha,x)[\ast,\cdot](\beta,y)=(\beta,y)\odot\{[(\beta,y)^{-1}\odot[(\alpha,x)]\odot(\beta,y)\}$$
$\Big(\mathcal{H},(\ast,\cdot)\Big)$ is called the left combined holomorph of $H(Q,\cdot)$ and $H'(Q,\ast)$.
\begin{myth}\label{davidthm9.1l}
Let $(Q,\cdot)$ be a left Bol loop with holomorph $(H,\odot)$ and  its corresponding middle Bol loop $(Q,\ast)$ with holomorph $(H, \circledast)$. If $(\mathcal{H},[\ast,\cdot])$ is the left combined holomorph of $H(Q,\cdot)$ and $H'(Q,\ast)$, then $(\mathcal{H},[\ast,\cdot])=(H', \circledast)$ if and only if $L^{-1}_yR_yL_y=L^{-1}_{y\phi}R_yL_{y\phi}$ if and only if $y^{-1}\phi\cdot y(zy)=(yz\cdot y^{-1}\phi)y~\forall ~y, z\in Q$ and $\phi\in A(Q,\cdot)$.
\end{myth}
\begin{proof}
Let $(Q,\cdot)$ be a left Bol loop and $(Q, \ast)$ its corresponding middle Bol loop, then by \eqref{36}, $x\ast y=y(y^{-1}x\cdot y)$.
Following the argument from \eqref{36} to \eqref{37}, we have
$$(H,\odot)=H(Q,\cdot):(\alpha,x)\odot(\beta,y)=(\alpha\beta, x\beta\cdot y)~\textrm{and}$$
 $$(H,\circledast)=H(Q,\ast): (\alpha,x)\circledast(\beta,y)=(\alpha\beta, x\beta\ast y)$$
$$=(\alpha\beta, y(y^{-1}x\beta\cdot y).$$
Keeping \eqref{35} in mind, then we have
$$(\alpha,x)[\ast,\cdot](\beta,y)=(\beta,y)\odot\{[(\beta,y)^{-1}\odot[(\alpha,x)]\odot(\beta,y)\}$$
$$=(\beta,y)\odot\bigg\{\Big(\beta^{-1}\alpha, (y^{-1})\beta^{-1}\alpha\cdot x\Big)\odot(\beta,y)\bigg\}$$
$$=\bigg(\alpha\beta, y\beta^{-1}\alpha\beta\cdot[(y^{-1})\beta^{-1}\alpha\beta\cdot x\beta]y\bigg).$$
$\therefore,
(\mathcal{H},[\ast,\cdot])=(H', \circledast)\Leftrightarrow (\alpha,x)[\ast,\cdot](\beta,y)=(\alpha,x)\circledast(\beta,y)\iff$
$$(\alpha\beta, y(y^{-1}x\beta\cdot y))=(\alpha\beta, y\beta^{-1}\alpha\beta\cdot[(y^{-1})\beta^{-1}\alpha\beta\cdot x\beta]y)\iff$$
$$y(y^{-1}x\beta\cdot y)=y\beta^{-1}\alpha\beta\cdot[(y^{-1})\beta^{-1}\alpha\beta\cdot x\beta]y)\iff$$
$$\beta L^{-1}_{y}R_yL_y=\beta L^{-1}_{y\beta^{-1}\alpha\beta}R_yL_{y\beta^{-1}\alpha\beta}\iff$$
$$L^{-1}_{y}R_yL_y=L^{-1}_{y\beta^{-1}\alpha\beta}R_yL_{y\beta^{-1}\alpha\beta}\iff$$
$$a L^{-1}_{y}R_yL_y= a L^{-1}_{y\beta^{-1}\alpha\beta}R_yL_{y\beta^{-1}\alpha\beta}, \forall a\in Q.$$
$$y((y^{-1}\cdot a)y)=(a\cdot y^{-1}\beta^{-1}\alpha\beta)y\cdot L_{y\beta^{-1}\alpha\beta}$$
$$y\cdot((y^{-1}\cdot a)y)=y\beta^{-1}\alpha\beta\cdot((a\cdot y^{-1}\beta^{-1}\alpha\beta)y).$$
Denote $\beta^{-1}\alpha\beta$ by $\phi$, we have;
$$y\cdot\Big((y^{-1}\cdot a)y\Big)=y\phi\cdot(a\cdot y^{-1}\phi)y.$$
\end{proof}

\end{document}